\documentclass[a4paper]{article}

\usepackage{amsthm,amsmath,amssymb}

\usepackage{graphicx}

\usepackage[colorlinks=true,citecolor=black,linkcolor=black,urlcolor=blue]{hyperref}

\newcommand{\arxiv}[1]{\href{http://arxiv.org/abs/#1}{\texttt{arXiv:#1}}}

\theoremstyle{plain}
\newtheorem{theorem}{Theorem}
\newtheorem{lemma}[theorem]{Lemma}
\newtheorem{corollary}[theorem]{Corollary}

\theoremstyle{definition}

\newtheorem{conjecture}[theorem]{Conjecture}

\newtheorem{question}[theorem]{Question}

\theoremstyle{remark}

\def\mc{\mathcal}
\def\F{\mc{F}}
\def\I{\mc{I}}

\def\F{\mathcal{F}}
\def\tr{\textrm}

\title{New Tur\'an densities for 3-graphs}

\author{Rahil Baber\\
\small 613 Greenford Road,\\ 
\small London, UK.\\
\small \tt rahilbaber@hotmail.com\\
\and John Talbot\thanks{This author is a Royal Society University Research Fellow}\\
\small Department of Mathematics\\
\small UCL, London, UK.\\
\small \tt j.talbot@ucl.ac.uk}

\date{\today}

\begin{document}
\maketitle

\begin{abstract}
If $\F$ is a family of graphs then the Tur\'an density of $\F$ is determined by the minimum chromatic number of the members of $\F$. 

The situation for Tur\'an densities of 3-graphs is far more complex and still very unclear. Our aim in this paper is to present new exact Tur\'an densities for individual and finite families of $3$-graphs, in many cases we are also able to give corresponding stability results. As well as providing new examples of individual $3$-graphs with Tur\'an densities equal to $2/9,4/9,5/9$ and $3/4$ we also give examples of irrational Tur\'an densities for finite families of 3-graphs, disproving a conjecture of Chung and Graham. (Pikhurko has independently disproved this conjecture by a very different method.)

A central question in this area, known as \emph{Tur\'an's problem}, is to determine the Tur\'an density of $K_4^{(3)}=\{123, 124, 134, 234\}$. Tur\'an conjectured that this should be $5/9$. Razborov [\emph{On 3-hypergraphs with forbidden 4-vertex configurations} in SIAM J. Disc. Math. {\bf 24}  (2010), 946--963] showed that if we consider the induced Tur\'an problem forbidding $K_4^{(3)}$ and $E_1$, the 3-graph with 4 vertices and a single edge, then the Tur\'an density is indeed $5/9$. We give some new non-induced results of a similar nature, in particular we show that $\pi(K_4^{(3)},H)=5/9$ for a $3$-graph $H$ satisfying $\pi(H)=3/4$. 

We end with a number of open questions focusing mainly on the topic of which values can occur as Tur\'an densities.

Our work is mainly computational, making use of Razborov's flag algebra framework. However all proofs are exact in the sense that they can be verified without the use of any floating point operations. Indeed all verifying computations use only integer operations, working either over $\mathbb{Q}$ or in the case of irrational Tur\'an densities over an appropriate quadratic extension of $\mathbb{Q}$.
\end{abstract}

\section{Introduction}
An \emph{$r$-graph} is a pair
$F=(V(F),E(F))$ where $V(F)$ is a set of \emph{vertices} and $E(F)$
is a family of $r$-subsets of $V(F)$ called \emph{edges}. So a $2$-graph is
a simple graph. For ease of notation we usually identify an
$r$-graph with its edge set. The number of edges in $F$ is denoted by $e(F)$.

Given a family of $r$-graphs $\F$ we say that an $r$-graph $H$ is
\emph{$\F$-free} if $H$ does not contain a subgraph isomorphic to
any member of $\F$. For any integer $n\geq 1$ we define the
\emph{Tur\'an number} of $\F$ to be
\[
\tr{ex}(n,\F)=\max\{e(H):H \tr{ is $\F$-free},\ |V(H)|=n\}.
\]
 Even in the simplest case of 2-graphs this parameter can be very difficult to determine exactly thus we will consider the related asymptotic density.

The \emph{Tur\'an density} of $\F$ is defined to be the following
limit (a simple averaging argument due to Katona, Nemetz and Simonovits \cite{KNS} shows that it always exists)
\[
\pi(\F)=\lim_{n\to\infty} \frac{\tr{ex}(n,\F)}{\binom{n}{r}}.\]

There are two general questions that are of interest to us. 
\begin{question}\label{qu:1} Given a family of $r$-graphs $\F$, what is $\pi(\F)$? 
\end{question}
\begin{question}\label{qu:2} Which values in $[0,1)$ are Tur\'an densities of families of $r$-graphs?
\end{question}
For $r=2$ the Erd\H os--Stone--Simonovits theorem answers both questions completely. 
\begin{theorem}[Erd\H os and Stone \cite{ES1}, Erd\H os and Simonovits \cite{ES2}]\label{es:thm}
Let $\F$ be a family of $2$-graphs. If $t=\min\{\chi(F):F\in\F\}\geq 2$ then \[
\pi(\F)=1-\frac{1}{t-1}.\] 
In particular the set of Tur\'an densities of 2-graphs is $\{0,1/2,2/3,3/4,\ldots\}$.
\end{theorem}
For $r\geq 3$ remarkably little is known. One general result for $r$-graphs is the following theorem of Erd\H os. An $r$-graph is \emph{$r$-partite} if its vertices can be partitioned into $r$ classes so that each edge meets each class exactly once. 
\begin{theorem}[Erd\H os \cite{E}]\label{erd:thm}
If $K^{(r)}(t)$ is the complete $r$-partite $r$-graph with $t$ vertices in each class then $\pi(K^{(r)}(t))=0$. 
\end{theorem}
Since $\lim_{t\to \infty}e(K^{(r)}(t))/\binom{tr}{r}=r!/r^r$ and all subgraphs of $K^{(r)}(t)$ are $r$-partite  we have the following simple corollary.
\begin{corollary}\label{erd:cor}
If $\F$ is a family of $r$-graphs then either at least one member of $\F$ is $r$-partite and so $\pi(\F)=0$, or none are $r$-partite and $\pi(\F)\geq r!/r^r$.
\end{corollary}

Essentially the only other general result is the following.
\begin{theorem}[Mubayi \cite{Mu} and Pikhurko \cite{Pi}]\label{mp:thm}
For $3\leq r\leq t$ let $H_t^r$ be the $r$-graph with vertices $x_i$ for $1\leq i\leq t$ and $y_{ij}^k$ for $1\leq i<j\leq t$, $1\leq k\leq r-2$ together with edges $x_ix_jy_{ij}^1\cdots y_{ij}^{r-2}$, for $1\leq i<j\leq t$.
\[
\pi(H_{t+1}^r)=\frac{r!}{t^r}\binom{t}{r}.\]
\end{theorem}

Attention has focused mainly on Question \ref{qu:1}, in particular a lot of work has gone into determining or giving bounds for the Tur\'an density of particularly simple 3-graphs such as $K_4^-=\{123, 124,134\}$, $K_4^{(3)}=\{123, 124,134,234\}$ and $F_{3,2}=\{123, 145, 245, 345\}$. 

In Section \ref{simple:sec} we give some new Tur\'an results for individual 3-graphs. In particular we give the first examples of single $3$-graphs with Tur\'an density $5/9$ for which Tur\'an's construction $T_n$ is asymptotically extremal (see Section \ref{5.9:sec} for definitions).

In Section \ref{finite:sec} we focus on Question \ref{qu:2}, in particular giving the first examples of irrational Tur\'an densities of finite families of $3$-graphs.  

We then return to the classical ``Tur\'an Problem'' of determining $\pi(K_4^{(3)})$. 

Given an exact Tur\'an density result for a family of $r$-graphs $\F$ there are two very natural questions one can ask. Firstly, what is the exact Tur\'an number $\tr{ex}(n,\F)$? Secondly, is there a ``stability'' result saying that all almost extremal $\F$-free $r$-graphs have essentially the same structure? Pikhurkho \cite{Pexact} answered both of these questions in the case of $\F=\{K_4^{(3)},E_1\}$, where $E_1$ is the $3$-graph with $4$ vertices and a single edge. We are able to give stability versions of all of our results from Sections \ref{simple:sec} and \ref{turan:sec}. Luckily once we have a ``flag algebra proof'' of the Tur\'an density of each family we can prove stability without having to consider each family in turn. Essentially  we prove one stability  result for each construction,  the details are given in Section \ref{stab:sec}.

Stability probably also holds for the results in Section \ref{finite:sec} but we have not proved this. The question of determining the exact Tur\'an number for each of the families we consider seems much more difficult although it is plausible that for all of our results $\tr{ex}(n,\F)$ is given by the corresponding construction for all sufficiently large $n$.

All discussion  of the computational proofs is deferred to the final section, with transcripts of the actual proofs forming a separate appendix. However we emphasise that all of our proofs can be verified using only integer operations and hence are genuine proofs rather than numerical results with the potential for rounding errors. These proofs are set in Razborov's flag algebra framework \cite{RF} and make heavy use of semi-definite programming (see \cite{BT} and \cite{R4}).

A key tool we will make use of is the ``blow-up'' of an $r$-graph. Given an $r$-graph $F$ and an integer $t\geq 1$ the \emph{blow-up} $F(t)$ is the $3$-graph formed by replacing each vertex of $F$ with a class of $t$ vertices and inserting a complete $r$-partite $r$-graph between any vertex classes corresponding to an edge in $F$. Given a family $\F=\{F_1,\ldots,F_s\}$ of $r$-graphs and an integer vector $\mathbf{t}=(t_1,\ldots,t_s)$ with each $t_i\geq 1$, we define the $\mathbf{t}$-\emph{blow-up} of $\F$ to be $\F(\mathbf{t})=\{F_i(t_i):1\leq i\leq s\}$.

The following result will be extremely useful.
\begin{theorem}[Brown and Simonovits \cite{BS}]\label{bs1:thm}
If $\F=\{F_1,\ldots,F_s\}$ is a family of $r$-graphs and $\mathbf{t}=(t_1,\ldots,t_s)$ is an integer vector with each $t_i\geq 1$ then $\pi(\F(\mathbf{t}))=\pi(\F)$.
\end{theorem}

In particular we have the following corollary that will often simplify the computations we perform. We will write $F\leq G$ to mean ``$F$ is contained in a blow-up of $G$''.
\begin{corollary}\label{sat:cor}
If $\F$ is a family of $r$-graphs and $G_1,G_2$ are $r$-graphs with $G_1\leq G_2$ then
\begin{itemize}\item[(i)] $\pi(\F\cup G_1)\leq \pi(\F\cup G_2)$,
\item[(ii)] $\pi(\F\cup G_1)=\pi(\F\cup G_1\cup G_2)$.
\end{itemize}
\end{corollary}
\begin{proof} Let $t\geq 1$ satisfy $G_1\subseteq G_2(t)$.
Theorem \ref{bs1:thm} implies that $\pi(\F\cup G_2(t))=\pi(\F\cup G_2)$. While  $G_1\subseteq G_2(t)$ implies that $\pi(\F\cup G_1)\leq \pi(\F\cup G_2(t))$. Hence (i) holds.

For (ii) we note that $\pi(\F\cup G_1)\geq \pi(\F\cup G_1\cup G_2)$ is trivial. While (i) implies that
\[
\pi(\F\cup G_1)=\pi(\F\cup G_1\cup G_1)\leq \pi(\F\cup G_1\cup G_2).\]
\end{proof}
For a detailed description of how we apply Corollary \ref{sat:cor} (to prove Theorem \ref{2.9:thm}) see the discussion in Section \ref{comp:sec}. Essentially we apply part (ii) repeatedly: if $F\leq G$ for each $G$ in some family $\mc{G}$ then $\pi(F)=\pi(F\cup \mc{G})$. This often leads to much more tractable computational problems.

When investigating new Tur\'an density results of $r$-graphs we have to be clear about what makes a result new. Consider the following situation: we have an $r$-graph $G$ whose Tur\'an density is known together with a sequence of asymptotically extremal examples $\{G_n\}_{n=1}^\infty$ (i.e.~$G_n$ is a $G$-free $r$-graph of order $n$  and $\lim_{n\to \infty} e(G_n)/\binom{n}{r}=\pi(G)$). Given a subgraph $F$ of $G$ we obviously know that $\pi(F)\leq \pi(G)$ and moreover if $G_n$ is $F$-free for all $n\geq 1$ then $\pi(F)=\pi(G)$. Corollary \ref{sat:cor} sometimes allows us to deduce new Tur\'an densities by checking for ``containment in blow-ups''. 
We will not be interested in results that are implied by known Tur\'an density results by taking blow-ups and applying Corollary \ref{sat:cor}. (See the remark following Theorem \ref{ff:thm} for an example of such a result.)

Note that checking if $F\leq G$ can be computationally difficult. For example suppose $r=2$ and $G=K_3$. Checking if a given graph $F$ satisfies $F\leq K_3$ is equivalent to determining whether $F$ is $3$-colourable: a well known NP-complete problem.

 An $r$-graph $F$ is said to be \emph{covering} if every pair of vertices from $V(F)$ belongs to an edge in $F$. For example, complete $r$-graphs are covering. 
Covering $r$-graphs are easier to deal with when checking containment in blow-ups.
\begin{lemma}\label{cover:lem}
If $F$ and $G$ are $r$-graphs and $F$ is covering then $F\leq G$ iff $F\subseteq G$.
\end{lemma}
\begin{proof}
If $F$ is a subgraph of $G(t)$ for some $t\geq 1$ then each vertex in $V(F)$ belongs to a different class in $G(t)$ (since there is an edge of $F$ containing any pair of vertices). Thus $F$ is a subgraph of $G$.\end{proof}
\section{Tur\'an densities of individual $3$-graphs}\label{simple:sec}
We require a couple of basic definitions. For an integer $n\geq 1$ let $[n]=\{1,2,\ldots,n\}$. If $[n]=A_1\cup A_2\cup \cdots\cup A_k$ is a partition then we say that it is \emph{balanced} if $||A_i|-|A_j||\leq 1$ for all $i,j\in [k]$.
\subsection{Density 2/9}
\begin{figure}[ht]
\begin{center}
\includegraphics[height=3cm]{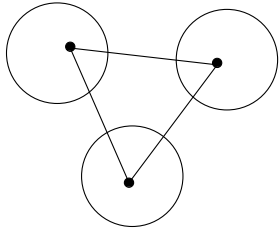}
\caption{The balanced complete tripartite $3$-graph, $S_n$.}\label{s3:fig}
\end{center}
\end{figure}
Given a tripartition $[n]=V_0\cup V_1 \cup V_2$ let $S(V_0,V_1,V_2)$ denote the \emph{complete tripartite} 3-graph with vertex classes $V_0,V_1,V_2$. Let $S_n$ denote the complete tripartite $3$-graph with the maximum number of edges (given by a balanced tripartition of $[n]$). Note that $e(S_n)=\lfloor \frac{n}{3}\rfloor \lfloor \frac{n+1}{3}\rfloor \lfloor \frac{n+2}{3}\rfloor$, and so $\lim_{n\to \infty} e(S_n)/\binom{n}{3}= 2/9$. Since all subgraphs of $S_n$ are tripartite this implies that any non-tripartite 3-graph $F$ satisfies $\pi(F)\geq 2/9$.

The first Tur\'an-type result for non-tripartite $3$-graphs was due to Bollob\'as \cite{B}. Let $K_4^-=\{123, 124,134\}$ and $F_5=\{123,124,345\}$.
\begin{theorem}[Bollob\'as \cite{B}]\label{b:thm} If $\F=\{K_4^-,F_5\}$ then $\tr{ex}(n,\F)=e(S_n)$. In particular $\pi(\F)=2/9$.\end{theorem}
This was followed by the Tur\'an result for the single $3$-graph $F_5$.
\begin{theorem}[Frankl and F\"uredi \cite{FF}]\label{ff:thm}
If $n\geq 3000$ then $\tr{ex}(n,F_5)=e(S_n)$. In particular $\pi(F_5)=2/9$.\end{theorem}
In fact, as a Tur\'an density result, Theorem \ref{ff:thm} does not meet our definition of a new result since Corollary \ref{sat:cor} (ii) allows us to deduce that $\pi(F_5)=2/9$ from Theorem \ref{b:thm}: take $\F=\emptyset$, $G_1=F_5$, $G_2=K_4^-$ and note that $F_5\subseteq K_4^-(2)$.

Theorem \ref{mp:thm} tells us that we also have $\pi(H_4^3)=2/9$, 
however this is implied by $\pi(F_5)=2/9$ and Theorem \ref{bs1:thm}, since $H_4^3 \subseteq F_5(3)$.

The following new result implies all of the aforementioned results.
\begin{theorem}\label{2.9:thm}
If $H=\{123,124,345,156\}$ then $\pi(H)=2/9$.
\end{theorem}
Note that $H$ is not contained in a blow-up of $F_5$ so this is a genuinely new result.

The proof of Theorem \ref{2.9:thm} uses Razborov's flag algebras framework \cite{RF}, \cite{R4} as well as Corollary \ref{sat:cor}. It is a straightforward calculation in this setting. For a general discussion of our methods see Section \ref{comp:sec}. A detailed computational proof can be found in the appendix file \texttt{2-9-prf.txt}. 
\subsection{Density 4/9}
\begin{figure}[ht]
\begin{center}
\includegraphics[height=3cm]{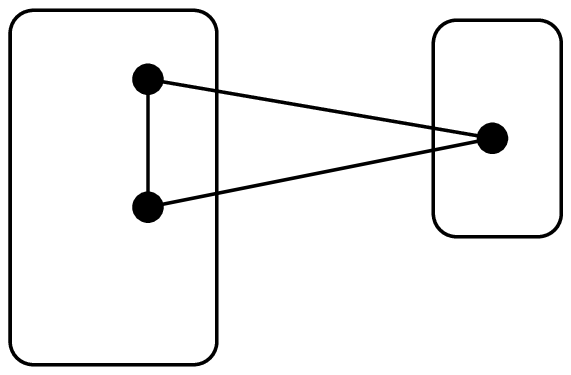}
\caption{A $(2,1)$-colourable  $3$-graph.}\label{hn:fig}
\end{center}
\end{figure}
Given a bipartition $[n]=V_0\cup V_1$ let $J(V_0,V_1)$ denote the 3-graph with vertex set $[n]$ and edges consisting of all triples meeting $V_0$ in two vertices and $V_1$ in one vertex. We call this the \emph{complete $(2,1)$-colourable $3$-graph} with classes $V_0$ and $V_1$. We say that a $3$-graph $G$ is \emph{$(2,1)$-colourable} if $G$ is isomorphic to a subgraph of $J(V_0,V_1)$ for some bipartition $[n]=V_0\cup V_1$. Let $J_n$ denote the $(2,1)$-colourable 3-graph of order $n$ with the maximum number of edges. A simple calculation shows that $J_n=J(V_0,V_1)$ for some bipartition with $|V_0|$ approximately twice as large as $|V_1|$ and so it is easy to check that $\lim_{n\to \infty} e(J_n)/\binom{n}{3}=4/9$. Hence any $3$-graph $F$ that is not $(2,1)$-colourable satisfies $\pi(F)\geq 4/9$. 

An example of a non-$(2,1)$-colourable  $3$-graph is $F_{3,2}=\{123,145,245,345\}$.
\begin{theorem}[F\"uredi, Pikhurko and Simonovits \cite{FPS}]\label{fps:thm}
For all $n\geq 3$ we have $\tr{ex}(n,F_{3,2})=e(J_n)=\max_{k}(n-k)\binom{k}{2}$. In particular $\pi(F_{3,2})=4/9$.\end{theorem}

We do not have an extension of this result, however we do have two new examples. 
\begin{theorem}\label{4.9:thm}
The $3$-graphs $G_1$ and $G_2$, given below, are non-$(2,1)$-colourable and satisfy $\pi(G_1)=\pi(G_2)=4/9$,
\[
G_1=\{123,124,134,235,245,156\},\qquad G_2=\{123,124,135,345,146,256\}.\]
\end{theorem}
We note that the three examples of $3$-graphs with Tur\'an density $4/9$: $F_{3,2}$, $G_1$ and $G_2$ are all incomparable under blow-ups. See Section \ref{comp:sec} for discussion. Detailed computational proofs can be found in the appendix files \texttt{4-9-01-prf.txt} and \texttt{4-9-02-prf.txt}.
\subsection{Density 5/9}\label{5.9:sec}
\begin{figure}[ht]
\begin{center}
\includegraphics[height=4cm]{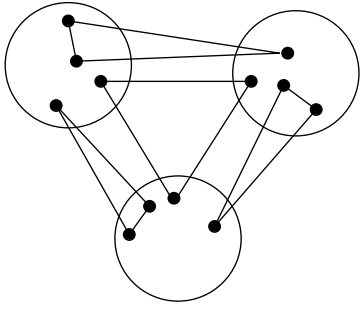}
\caption{Tur\'an's construction: $T_n$.}\label{t3:fig}
\end{center}
\end{figure}
One obvious sequence of $3$-graphs with asymptotic density $5/9$ is given by taking the balanced blow-ups of $K_6^{(3)}$, the complete $3$-graph of order $6$. If $n$ is a multiple of six then $K_6^{(3)}(n/6)$ has $20(n/6)^3$ edges. This construction is extremal (at least asymptotically) for the $3$-graph $H_7^3$ (see Theorem \ref{mp:thm}), and so $\pi(H_7^3)=5/9$. 

Another sequence of $3$-graphs with asymptotic density $5/9$ was first introduced by Tur\'an. Given a tripartition $[n]=V_0\cup V_1\cup V_2$ define the $3$-graph $T(V_0,V_1,V_2)$ to have as edges all triples meeting each $V_i$ exactly once together with those triples containing two vertices from $V_i$ and one from $V_{i+1}$ (where subscripts are understood modulo 3). If the tripartition is balanced then we denote this $3$-graph by $T_n$ (again a simple calculation shows that $T_n$ has the maximum number of edges of all such 3-graphs).

Tur\'an conjectured that $\tr{ex}(K_4^{(3)},n)=e(T_n)$ and hence $\pi(K_4^{(3)})=5/9$. This conjecture is still far from resolved and we will return to it in Section \ref{turan:sec}. For now it is sufficient to note that previously there were no known examples of single $3$-graphs $F$ satisfying $\pi(F)=5/9$, with the lower bound provided by $T_n$. (Since $H_t^3$ is $(2,1)$-colourable for any $t$ we have $H_t^3\subseteq T_n$ for $n$ sufficiently large. In particular $H_7^3\subseteq T_n$.)

\begin{theorem}\label{5.9:thm}
Each $3$-graph $F$ in the list below satisfies $\pi(F)=5/9$ and $T_n$ is $F$-free. These $3$-graphs are all incomparable with respect to blow-ups.
\begin{align*}
&\{123, 124, 134, 125, 245, 136, 346, 156\},\\
&\{123, 124, 134, 125, 135, 245, 345, 236, 456\},\\
&\{123, 124, 134, 125, 135, 245, 126, 236, 146\},\\
&\{123, 124, 134, 125, 135, 345, 126, 236, 246\},\\
&\{123, 124, 134, 125, 235, 345, 126, 246, 156\},\\
&\{123, 124, 134, 125, 235, 136, 346, 156, 356\},\\
&\{123, 124, 134, 125, 135, 245, 126, 136, 346, 456\},\\
&\{123, 124, 134, 125, 135, 345, 126, 236, 146, 156\},\\
&\{123, 124, 134, 125, 135, 245, 126, 236, 346, 356\},\\
&\{123, 124, 134, 125, 135, 345, 126, 236, 346, 356\},\\
&\{123, 124, 134, 125, 135, 146, 246, 156, 256, 456\},\\
&\{123, 124, 134, 125, 135, 146, 246, 156, 356, 456\}.
\end{align*}
\end{theorem}
See Section \ref{comp:sec} for discussion of our proof methods, again we made extensive use of Corollary \ref{sat:cor} (ii). Detailed computational proofs can be found in the appendix files \texttt{5-9-01-prf.txt} to \texttt{5-9-12-prf.txt}.
\subsection{Density 3/4}
\begin{figure}[ht]
\begin{center}
\includegraphics[height=3cm]{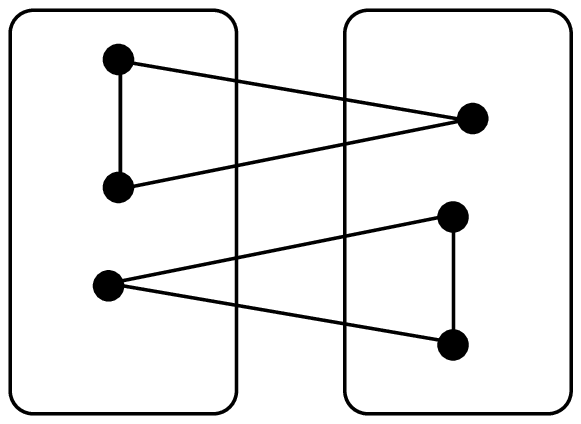}
\caption{A bipartite $3$-graph.}\label{bn:fig}
\end{center}
\end{figure}
We say that a $3$-graph is \emph{bipartite} if there is a partition of its vertex set into two classes, neither of which contains an edge. Given a bipartition $[n]=V_0\cup V_1$ let $B(V_0,V_1)$ be the \emph{complete bipartite} 3-graph with vertex classes $V_0$ and $V_1$, i.e.~its edges are all triples meeting both $V_0$ and $V_1$. If the bipartition is balanced then we denote this 3-graph by $B_n$. Clearly $B_n$ is a bipartite $3$-graph of order $n$ with the maximum number of edges. Moreover $\lim_{n\to \infty}e(B_n)/\binom{n}{3}=3/4$ and so any non-bipartite $3$-graph $F$ satisfies $\pi(F)\geq 3/4$.

The first example of a $3$-graph with Tur\'an density $3/4$ was given by de Caen and F\"uredi \cite{CF}, proving a conjecture of S\'os. The \emph{Fano plane} is the $3$-graph $PG(2,2)=\{123, 145, 356, 167, 257, 347, 246\}$.
\begin{theorem}[de Caen and F\"uredi \cite{CF}]\label{cf:thm}
The Fano Plane $PG(2,2)$ satisfies $\pi(PG(2,2))=3/4$.
\end{theorem}

Their method was extended by Mubayi and R\"odl \cite{MR} to show that a number of other $3$-graphs have Tur\'an density $3/4$.

For $p,q\geq 1$ let  $F_{p,q}$ be the 3-graph with vertex set $[p+q]$ and edges $\binom{[p]}{3}\cup \{xyz:x\in [p],y,z\in [p+q]-[p]\}$.  Let $F'_{3,3}$ be a copy of $F_{3,3}$ with  two additional vertices, 7, 8, and four additional edges 178, 278, 478, 578. Let $F''_{3,3}$ be obtained from $F'_{3,3}$ by adding two new vertices, 9,$a$, and three edges, $19a$, $49a$, $79a$. Let $F^-_{4,3}$ be the 3-graph obtained from $F_{4,3}$ by deleting the edge 156. Let
$F'^{-}_{4, 3}$ be obtained from $F^-_{4,3}$ by adding two vertices 8, 9, and adding three edges 289, 389, 589.

\begin{theorem}[Mubayi and R\"odl \cite{MR}]\label{mr:thm}
Let $\mc{S} = \{F_{3,3}, F_{3,3}',F_{3,3}'', F_{4, 3}^{-}, F'^{-}_{4, 3}\}$. If $A, B \in \mc{S}$ and $A\subseteq F \subseteq B$ then $\pi(F) = 3/4$.
\end{theorem}
We have the following new results.

\begin{theorem}\label{3.4:thm}
Each $3$-graph $F$ in the list below satisfies $\pi(F)=3/4$ and $B_n$ is $F$-free. All of these $3$-graphs are incomparable with respect to blow-ups.
\begin{align*}
&\{123, 124, 134, 234, 125, 135, 235, 145, 126, 136, 236, 146, 256, 356\},\\
&\{123, 124, 134, 234, 125, 135, 235, 145, 245, 126, 136, 236, 146, 356, 456\},\\
&\{123, 124, 134, 234, 125, 135, 235, 145, 245, 126, 136, 146, 346, 256, 356, 456\}.
\end{align*}
\end{theorem}
We remark that these are genuinely new results: the three $3$-graphs in Theorem \ref{3.4:thm} are all covering and are not contained in any of the $3$-graphs listed in Theorems \ref{cf:thm} and \ref{mr:thm} (thus by Lemma \ref{cover:lem} we cannot deduce their Tur\'an densities via blow-ups from the earlier results). 

%
%

See Section \ref{comp:sec} for discussion of our proof methods. Detailed computational proofs can be found in the appendix files \texttt{3-4-01-prf.txt} to \texttt{3-4-03-prf.txt}.

We note that for all of our new results (Theorems \ref{2.9:thm}, \ref{4.9:thm}, \ref{5.9:thm}, \ref{3.4:thm}) we have corresponding stability theorems (we defer a discussion of stability and exact Tur\'an numbers to Section \ref{stab:sec}). 

\section{New Tur\'an densities for finite families}\label{finite:sec}
In the previous section we focused almost exclusively on Tur\'an densities of individual $3$-graphs. We now turn to the question of which values from $[0,1)$ can occur as Tur\'an densities of families of $r$-graphs (Question \ref{qu:2}). We will be interested in the size of the families in question and so require the following definitions. 

For $r\geq 3$ and $t\geq 1$ integers, define \[
\Pi^{(r)}_t=\{\pi(\F):\tr{$\F$ is a family of $r$-graphs and $|\F|\leq t$}\},\] \[
\Pi^{(r)}_{\infty}=\{\pi(\F): \F\tr{ is a family $\F$ of $r$-graphs}\}\]
and \[
\Pi^{(r)}_{\tr{fin}}=\{\pi(\F):\tr{$\F$ is a finite family of $r$-graphs}\}.\]
Obviously the following containments hold:
\[
\Pi^{(r)}_1\subseteq \Pi^{(r)}_2\subseteq \cdots \subseteq \Pi^{(r)}_{\tr{fin}}\subseteq \Pi^{(r)}_{\infty}.\]
To the best of our knowledge it is not known whether any of these containments are strict.

The two general results we have regarding these sets of densities are Erd\H os's result for $r$-partite $r$-graphs (Theorem \ref{erd:thm}) and Mubayi and Pikhurko's result for $H_t^r$ (Theorem \ref{mp:thm}). Putting these together yields.
\begin{corollary}\label{erd2:cor}
For all $t\geq r\geq 2$
\[
[0,\frac{r!}{r^r})\cap \Pi_1^{(r)}=[0,\frac{r!}{r^r})\cap \Pi_\infty^{(r)}=\{0\}\]
and
\[
\frac{r!}{t^r}\binom{t}{r}\in \Pi_1^{(r)}.\]
\end{corollary} 

A  useful tool, when searching for new Tur\'an densities, is the \emph{Lagrangian} of an $r$-graph. Let  $F$ be an $r$-graph with vertex set $[n]=\{1,2,\ldots,n\}$. Define \[ \Delta_n=\{(x_1,\ldots,x_n)\in \mathbb{R}^n:
\sum_{i=1}^n x_i=1,x_i\geq 0\}.\] For $x\in \Delta_n$ let
\[
\lambda(F,x)=r!\sum_{\{i_1,i_2,\ldots,i_r\}\in F}\prod_{j=1}^rx_{i_j}.\]
The \emph{Lagrangian} of $F$ is $\lambda(F)=\max_{x\in \Delta_n}\lambda(F,x)$.

The Lagrangian of an $r$-graph is closely related to certain blow-ups of $F$: it tells us how dense the densest blow-up of a subgraph of $F$ can be. We introduce a new set of densities:
\[
\Lambda^{(r)}=\{\lambda(F): F\tr{ is an $r$-graph}\}.\]

Brown and Simonovits showed that the following containments hold. (Note that for $A\subseteq \mathbb{R}$ we denote the closure of $A$ by $\overline{A}$.)
\begin{theorem}[Brown and Simonovits \cite{BS}]\label{bs2:thm}
If $r \geq 2$ then
\[
\Lambda^{(r)}\subseteq \Pi^{(r)}_{\infty}=\overline{\Pi}_{\tr{fin}}^{(r)}=\overline{\Lambda}^{(r)}.\]
\end{theorem}
Thus in particular every Lagrangian of an $r$-graph is the Tur\'an density of a corresponding infinite family of $r$-graphs. (In fact it is easy to see that for any $r$-graph $F$, the family $\F_F=\{G:G\tr{ is an $r$-graph with }\lambda(G)> \lambda(F)\}$ satisfies $\pi(\F_F)=\lambda(F)$.)

For small $3$-graphs it is straightforward to calculate the Lagrangian directly. We will use this to give some new examples of Tur\'an densities of finite families of $3$-graphs. In particular we give the first examples of irrational Tur\'an densities for finite families, disproving the following conjecture of Chung and Graham \cite{CG}.
\begin{conjecture}[Chung and Graham \cite{CG} pg 95]\label{CG:conj}
If $\mathcal{F}$ is a finite family of $r$-graphs then $\pi(\F)$ is rational.
\end{conjecture}
Pikhurko \cite{PP} has also disproved this conjecture (for all $r \geq 3$). His proof is  very different and the finite families he obtains are rather large.

We introduce the following notation: given an $r$-graph $G$ with vertex set $[k]$, a vector $x=(x_1,\ldots,x_k)\in \Delta_k$ and a large integer $n$; we define $G(x,n)$, the \emph{$n$ vertex blow-up of $G$ by $x$} to be the blow-up of $G$ in which vertex $i$ is replaced by a class of $\lfloor x_i n\rfloor $ vertices for $1\leq i \leq k-1$ and vertex $k$ is replaced by a class of size $n-\sum_{i=1}^{k-1}\lfloor x_i n\rfloor$. (Thus if $\lambda(G)=\lambda(G,x)$ then $\lim_{n\to \infty}e(G(x,n))/\binom{n}{r}=\lambda(G)$.)

\begin{theorem}\label{fin:thm}
In each case below the finite family of $3$-graphs $\F_i$ together with the $3$-graph $G_i$ and weighting $x$ satisfy: $G_i(x,n)$ is $\F_i$-free for all $n\geq 1$ and $\pi(\F_i)=\lambda(G_i)$.

\begin{align*}
&G_1=\{123,124,125,345\}=F_{3,2}\\
&\lambda(G_1)=\frac{189+15 \sqrt{5}}{961},\quad x_1=x_2=\frac{13+3\sqrt{5}}{62} ,x_3=x_4=x_5=\frac{6-\sqrt{5}}{31} .\\
&\F_1=\{\{123, 124, 135, 146, 156\},\{123, 124, 156, 346, 257\},\{123, 124, 156, 347, 567\}\}.\\
\\
&G_2=\{123, 234,345,145,125\}=C_5,\\
&\lambda(G_2)=\frac{6}{25},\quad x_1=x_2=x_3=x_4=x_5=\frac{1}{5}.\\
&\F_2=\{\{123, 124, 134\}, \{123, 124, 125, 345\}, \{123, 124, 135, 256, 167, 467\}\}.\\
\\
&G_3=\{123, 124, 134\}=K_4^-\\
&\lambda(G_3)=8/27,\quad x_1=\frac{1}{3},x_2=x_3=x_4=\frac{2}{9}.\\
&\F_3=\{\{123, 124, 134, 234\},\{123, 124, 125, 345, 346\},\{123, 124, 345, 156, 256\}, \{123, 124, 125, 346, 356, 456\}\}.\\
\\
&G_4=\{123, 124,125,134,135,145\}=F_{1,4}\\
&\lambda(G_3)=\frac{1}{3},\quad x_1=\frac{1}{3},x_2=x_3=x_4=x_5=\frac{1}{6}.\\ 
&\F_4=\{\{123, 124, 134, 156, 256\}, \{123, 124, 134, 125, 126, 357, 367, 457, 467, 567\}, \{123, 124, 345, 156, 257\}\}\\
\\
&G_5=\{123, 124, 125, 126, 134, 135, 146, 235, 246, 256, 345, 346, 356, 456\}=K_6^{(3)}\setminus C_6,\\
&\lambda(G_5)=\frac{7}{18},\quad x_1=x_2=x_3=x_4=x_5=x_6=\frac{1}{6}.\\
&\F_5=\{\{123, 124, 135, 145, 346, 256\},\{123, 124, 134, 125, 345, 136, 246\},
\{123, 124, 134, 125, 135, 126, 136,\\& 456\}\}.
\end{align*}
\begin{align*}
&G_6=\{123,124,134,234,135,235,145,245\}=K_5^{(3)}\setminus\{125,345\} \\
&\lambda(G_6)=\frac{32}{81},\quad x_1=x_2=x_3=x_4=\frac{2}{9},x_5=\frac{1}{9}.\\
&\F_6=\{ \{123, 124, 125, 346, 356, 456\},\{123, 124, 135, 256, 346, 456\}, \{123, 124, 135, 145, 256, 346\},\{123, 124,\\&134, 125, 126, 356, 456\}, \{123, 124, 134, 125, 135, 126, 136, 456\},\{123, 124, 134, 125, 135, 245, 146, 246, 256\},\\& \{123, 124, 134, 125, 135, 345, 126, 146, 346\},\{123, 124, 134, 125, 135, 235, 245, 146, 246\}\}.\\
\\
&G_7=\{123, 124,134, 234, 125,135,235,145,245\}=K_5^{(3)}\setminus\{345\},\\ 
&\lambda(G_7)=\frac{-35+13 \sqrt{13}}{27},\quad x_1=x_2=\frac{5-\sqrt{13}}{6}  ,x_3=x_4=x_5=\frac{-2+\sqrt{13}}{9} .\\
&\F_7=\{\{123, 124, 135, 345, 146, 256, 346\},\{123, 124, 134, 125, 135, 126, 136, 456\},\{123, 124, 134, 125, 136, \\&256,356, 456\},\{123, 124, 134, 125, 135, 145, 126, 136, 146, 156\},\{123, 124, 134, 234, 125, 135, 245, 236, 146,\\&346\}\}.
\end{align*}
\end{theorem}
In particular we have the following corollary.
\begin{corollary}
We have the following new Tur\'an densities for finite families of $3$-graphs
\[
\left\{\frac{189+15\sqrt{5}}{961},\frac{6}{25},\frac{8}{27},\frac{1}{3},\frac{7}{18},\frac{32}{81},\frac{13\sqrt{13}-35}{27}\right\}\subseteq \Pi_{\tr{fin}}^{(3)}.\]
\end{corollary}
See Section \ref{comp:sec} for discussion of our proof methods. Detailed computational proofs can be found in the appendix files \texttt{Root5-prf.txt, 6-25-prf.txt, 8-27-prf.txt, 1-3-prf.txt, 7-18-prf.txt, 32-81-prf.txt} and \texttt{Root13-prf.txt}.
\section{Tur\'an's problem}\label{turan:sec}
Tur\'an famously conjectured that $\tr{ex}(n,K_4^{(3)})=e(T_n)$, thus in particular $\pi(K_4^{(3)})=5/9$. If Tur\'an's conjecture is true then there are in fact exponentially many non-isomorphic extremal examples of $K_4^{(3)}$ $3$-graphs with $\tr{ex}(n,K_4^{(3)})$ edges (as described by Kostochka \cite{K} and Fon-der-Flass \cite{FdF}). 

Much work has been done on  Tur\'an's conjecture and recently Razborov gave the upper bound of $\pi(K_4^{(3)})\leq 0.561666$. Moreover, he noted that since Tur\'an's construction $T_n$ contains no set of four vertices inducing a single edge it is natural to consider a related ``induced Tur\'an'' problem. 

Given a family of $r$-graphs $\F$ we say that an $r$-graph $G$ is \emph{induced $\F$-free} if $G$ has no induced subgraph isomorphic to a member of $\F$. We then define the \emph{induced Tur\'an number} of a family of $r$-graphs $\F$ to be
\[
\tr{ex}_{\tr{ind}}(n,\F)=\max\{e(G):G\tr{ is an induced $\F$-free $r$-graph of order $n$}\}.\]
The \emph{induced Tur\'an density of $\F$} is then\[
\pi_{\tr{ind}}(\F)=\lim_{n\to \infty}\frac{ \tr{ex}_{\tr{ind}}(n,\F)}{\binom{n}{r}}.\]
\begin{theorem}[Razborov \cite{R4}]\label{raz:thm}
If $E_1$ is the $3$-graph with 4 vertices and 1 edge then $\pi_{\tr{ind}}(K_{4}^{(3)},E_1)=5/9$.
\end{theorem}
Our aim in this section is to give some non-induced results of a similar nature. In particular we have an example of a $3$-graph $H$ satisfying $\pi(H)=3/4$ and $\pi(K_4^{(3)},H)=5/9$.
\begin{theorem}\label{k4:thm}
Each of the $3$-graphs $H_i$ listed below satisfies $\pi(K_4^{(3)},H_i)=5/9$. They are all incomparable with respect to blow-ups. (For reference we also note the numerical upper bounds we found for $\pi(H_i)$ in each case).
\begin{align*}
 &H_1 =\{123, 124, 134, 125, 135 ,245, 345, 126,236, 146, 156,  456\}\quad \pi(H_1)=3/4,\\
&H_2 =\{123, 124, 134, 125, 135, 126, 236, 146, 346, 356, 456\}\quad  \pi(H_2)\leq 0.613,\\
&H_3 =\{123, 124, 134, 125, 135, 245,  345, 126, 236, 346, 356\}\quad  \pi(H_3)\leq 0.613,\\
&H_4 =\{123, 124, 134, 125, 135, 245, 236, 146, 346, 156, 456\}\quad  \pi(H_4)\leq 0.608,\\
&H_5 =\{123, 124, 134, 125, 135, 245, 345, 236, 146, 256, 456\}\quad  \pi(H_5)\leq 0.608,\\
&H_6 =\{123, 124, 134, 125, 135, 245, 345, 126, 136, 246, 346, 456\}\quad  \pi(H_6)\leq 0.597,\\
&H_7 =\{123, 124, 134, 125, 345, 136, 246, 256, 356, 456\}\quad  \pi(H_7)\leq 0.595,\\
&H_8 =\{123, 124, 134, 125, 135, 236, 146, 246, 156, 256, 456\}\quad  \pi(H_8)\leq 0.594,\\
&H_9 =\{123, 124, 134, 125, 135, 245, 236, 346, 356, 456\}\quad  \pi(H_9)\leq 0.555566,\\
&H_{10} =\{123, 124, 134, 125, 135, 245, 236, 246, 346, 456\}\quad  \pi(H_{10})\leq 0.55555557.
\end{align*}
\end{theorem}
The most interesting case of this result is $H_1$ so we focus on that now. It is straightforward to check that $H_1$ has a subgraph isomorphic to $F_{3,3}$ and hence is not bipartite. Moreover $H_1$ is a subgraph of the second example from Theorem \ref{3.4:thm}, hence $\pi(H_1)=3/4$. The proof of the second part of this theorem is again computational and can be found in the appendix, although a few remarks are in place. 

Our proof of Theorem \ref{k4:thm} mimics that of Razborov's proof of Theorem \ref{raz:thm}. In particular the flag algebra computation we perform makes use only of information contained in the $\{K_4^{(3)},H_1\}$-free $3$-graphs of order 6. (In Razborov's case these are replaced by the $K_4^{(3)}$-free $3$-graphs of order 6 with no induced $E_1$.) There are precisely 964 non-isomorphic $K_4^{(3)}$-free $3$-graphs of order 6. Of these exactly 34 do not contain an induced $E_1$ and thus play a role in Razborov's proof of Theorem \ref{raz:thm}. However it turns out that 962 of the $K_4^{(3)}$-free 3-graphs are $H_1$-free and are thus considered in the proof of Theorem \ref{k4:thm}.

See the final section for discussion of our proof methods. Detailed computational proofs of the results in this section can be found in the appendix files \texttt{K4+H-01-prf.txt} to \texttt{K4+H-10-prf.txt}.

A natural question to ask is whether any of the other $K_4^{(3)}$-free 3-graphs with $e(T_n)$ edges (described by Kostochka \cite{K}) are also $H_i$-free for each $i$. Although some of the other constructions are $H_i$-free for some $i$ we have a stability result saying that all almost extremal examples of $\{K_4^{(3)},H_i\}$-free 3-graphs have essentially the same structure as Tur\'an's construction $T_n$. See the next section for details.
\section{Stability and exactness}\label{stab:sec}
Given a family of $r$-graphs $\F$, we call a sequence of $r$-graphs $\{G_n\}_{n=1}^\infty$ \emph{almost extremal for $\F$} if each $G_n$ is $\F$-free of order $n$ with $d(G_n)=\pi(\F)+o(1)$.

\begin{theorem}[Stability]\label{stab:thm} Let $\F$ be one of the families of 3-graphs whose Tur\'an density is determined in Theorem \ref{2.9:thm}, \ref{4.9:thm}, \ref{5.9:thm}, \ref{3.4:thm} or \ref{k4:thm} and let $C_n\in \{S_n,J_n,T_n,B_n\}$ be the corresponding $\F$-free $3$-graph with density $d(C_n)=\pi(\F)+o(1)$. If $\{G_n\}_{n=1}^\infty$ is almost extremal for  $\F$ then we can make $G_n$ isomorphic to $C_n$ by changing at most $o(n^3)$ edges.
\end{theorem}

It turns out that some cases of Theorem \ref{stab:thm} require a little more work to prove than others. We give the proof in the ``easy'' case and then indicate how the other cases can be proved. 

Our proof follows a similar argument to that given for the family $\F=\{K_4^{(3)},E_1\}$ by Pikhurko (Theorem 2 \cite{Pexact}), although fortunately we can use the fact that each of the constructions we consider is characterised by its small induced subgraphs to avoid proving a separate result for each family (see Lemma \ref{ind:lem}). 

If $G$ is an $r$-graph let $\I_k(G)=\{G[A]:A\subseteq V(G), |A|=k\}$ be the set of all $k$-vertex induced subgraphs of $G$. Given another $r$-graph $H$ we define $p(H;G)$ to be the \emph{induced density} of $H$ in $G$: this is the probability that if $A\subseteq V(G)$ is a set of $|V(H)|$ vertices chosen uniformly at random then the subgraph induced by $A$ is isomorphic to $H$. 

Let $\mathcal{H}_k(\F)$ be the family of all $\F$-free 3-graphs of order $k$ up to isomorphism. We say that $H\in \mathcal{H}_k(\F)$ is $\mathcal{F}$\emph{-sharp} if there exists an almost extremal sequence $\{G_n\}_{n=1}^\infty$ for $\mathcal{F}$ such that $p(H;G_n)\neq o(1)$. If $H\in\mathcal{H}_k(\F)$ is not $\mathcal{F}$-sharp we say it is \emph{$\mathcal{F}$-negligible}. We denote the family of $\mathcal{F}$-sharp 
3-graphs of order $k$ by $\mathcal{H}_k^\#(\F)$.

To motivate our next result consider the following trivial fact. If $G$ is a $2$-graph of order at least 3, with the property that all induced subgraphs of $G$ of order $3$  are complete bipartite graphs then $G$ itself is a complete bipartite graph. In fact analogous results hold for the $3$-graph properties we are interested in.

We say that an $r$-graph property $\mc{P}$ is \emph{$k$-induced} if for any $r$-graph $G$ of order at least $k$, $\I_k(G)\subseteq \mc{P}\implies G\in \mc{P}$.  
\begin{lemma}\label{ind:lem}
The following 3-graph properties are all $6$-induced 
\begin{align*}
\mc{P}_S&=\{G:G\tr{ is a complete tripartite 3-graph}\},\\ 
\mc{P}_J&=\{G:G\tr{ is a complete $(2,1)$-colourable 3-graph}\},\\
\mc{P}_B&=\{G:G\tr{ is a complete bipartite 3-graph}\}.
\end{align*}
\end{lemma}
\begin{proof}[Proof of Theorem \ref{stab:thm}:]
Let $\F$ be one of the families given in the statement of Theorem \ref{stab:thm} with corresponding extremal construction $C_n\in\{S_n,J_n,T_n,B_n\}$. Suppose that $\{G_n\}_{n=1}^\infty$ is almost extremal for $\F$.


Our flag algebra proof of the Tur\'an density of $\mathcal{F}$ using $6$-vertex 3-graphs also provides us with information about $\mc{H}_6^\#(\F)$.  The associated proof file, e.g.~\texttt{5-9-05-prf.txt}, contains a list of all $6$-vertex $3$-graphs that potentially belong to $\mc{H}_6^\#(\F)$. In the easy case (which we now assume we are in) this tells us that $\mc{H}^\#_6(\F)\subseteq \I_6(C_n)$, i.e.~the only induced 6-vertex subgraphs that can occur with positive induced density in $G_n$ are those that are found in the corresponding construction $C_n$. 

If $\F$ is a family whose Tur\'an density is $5/9$ then Pikhurko's stability theorem (Theorem 2 \cite{Pi}) for $\{K_4^{(3)},E_1\}$ in fact also applies to $\F$, so let us suppose we have a family $\F$ whose Tur\'an density is $2/9,4/9$ or $3/4$, determined in Theorem \ref{2.9:thm}, \ref{4.9:thm}, or \ref{3.4:thm}.

We can now apply the hypergraph removal lemma of R\"odl and Schacht \cite{RS} and obtain a new sequence of $3$-graphs $\{G_n'\}_{n=1}^\infty$ satisfying $\I_6(G_n')\subseteq \I_6(C_n)$ by changing $o(n^3)$ edges. Thus, by Lemma \ref{ind:lem}, we know that $G_n'$ is isomorphic to $C(V_0,V_1,V_2)$ for some partition $[n]=V_0\cup V_1\cup V_2$. The result then follows by elementary calculus since $e(C(V_0,V_1,V_2))=e(G'_n)=e(C_n)+o(n^3)$ implies that the partition $V_0,V_1,V_2$ must be approximately that giving $C_n$ and hence by changing at most $o(n^3)$ edges in $G_n'$ we can obtain $C_n$.

This completes the proof in the easy case when our flag algebra proof tells us that $\mc{H}^\#_6(\F)\subseteq\I_6(C_n)$.  (This applies to the families \texttt{3-4-02, 5-9-05, 5-9-07, 5-9-08, 5-9-09, 5-9-10, 5-9-12, K4+H-01, K4+H+06}.)

There are two slightly more complicated cases:
\begin{itemize}
\item[(1)] Rather than determining $\pi(\F)$ directly we made use of blow-ups and Corollary \ref{sat:cor}. (This applies to the families \texttt{2-9, 4-9-01, 4-9-02, 5-9-01, 5-9-02, 5-9-03, 5-9-04, 5-9-06, 5-9-11}.)
\item[(2)] Our flag algebra proof does not immediately give $\mc{H}_6^\#(\F)\subseteq \I_6(C_n)$. (This applies to the families \texttt{3-4-01, 3-4-03, 5-9-01, 5-9-06, K4+H-02, K4+H+03, K4+H-04, K4+H+05, K4+H-07, K4+H+08, K4+H-09, K4+H+10}.)
\end{itemize}
We can deal with (1) as follows. If we used Corollary \ref{sat:cor} to determine $\pi(\F)$ then we have an auxillary family $\F'$ and a flag algebra proof determining $\pi(\F')$. Moreover for each $F' \in \F'$ there exists $F\in \F$ such that $F\leq F'$. Now suppose that $\{G_n\}_{n=1}^\infty$ is an almost extremal sequence for $\F$. We need to show that by changing at most $o(n^3)$ edges in $G_n$ we can obtain a sequence of $3$-graphs $\{G_n'\}_{n=1}^\infty$ that is almost extremal for $\F'$ (this will return us to the easy case of the proof). We can do this using the hypergraph removal lemma as long as we know that all 6 vertex 3-graphs $H$  that are not $\F'$-free satisfy $p(H;G_n)=o(1)$. This is straightforward to prove. Suppose there exists a 6 vertex 3-graph $H$ such that $p(H;G_n)\neq o(1)$ and $H$ is not $\F'$-free. Then there exist $F\in \F, F'\in\F', t\geq 1$ such that $F'$ is a subgraph of $H$ and $F$ is a subgraph of $F'(t)$. Moreover there exists $\epsilon>0$ and a subsequence $\{G_{n_k}\}_{k=1}^\infty$ such that $p(H;G_{n_k})\geq \epsilon$ for all $k$. Since $H$ contains $F'$, a standard ``supersaturation'' argument implies that $G_{n_k}$ contains arbitrarily large blow-ups of $F'$. Hence for $k$ large, $G_{n_k}$ contains $F'(t)$ and so is not $\F$-free, a contradiction.

To deal with complication (2) we have to show that the extra potentially $\F$-sharp 3-graphs given by our flag algebra proof are in fact $\F$-negligible. We omit the the details of this argument since it is tedious but not difficult. (In fact there are only two subcases to deal with: the spurious $\F$-sharp graphs are identical in the cases of \texttt{3-4-01} and  \texttt{3-4-03} and are also identical for the remaining families.)
\end{proof}

If $G$ is a 3-graph we say that $a,b\in V(G)$ are \emph{twins} if for all $x,y\in V\setminus\{a,b\}$ we have $axy\in G$ iff $bxy\in G$. 
\begin{proof}[Proof of Lemma \ref{ind:lem}]
In our arguments below we will be repeatedly examining small induced subgraphs of a given 3-graph $G$. If $a_1,\ldots,a_k\in V(G)$ then $G[a_1a_2\cdots a_k]$ is the subgraph induced by $\{a_1,\ldots,a_k\}$. Note that the vertices need not be distinct and so the induced subgraph may have less than $k$ vertices.

We first sketch the proof for $\mc{P}_S$. Let $G$ be a 3-graph of order $n\geq 6$ satisfying $\I_6(G)\subseteq \mc{P}_S$. Define a relation $\sim_S$ on $V=V(G)$ by $a\sim_S b$ iff $a=b$ or there exist distinct $c,d\in V\setminus\{a,b\}$ such that $G[abcd]=S(ab,c,d)=\{acd,bcd\}$, in which case we say $a\sim_S b$ \emph{via} $cd$. We claim that this defines an equivalence relation on $V$. We need to check transitivity. Suppose $a\sim_S b$ via $uv$ and $b\sim_S c$ via $xy$, then without loss of generality $G[abuvxy]=S(ab,ux,vy)$ so $G[acxy]=S(ac,x,y)$ and $a\sim_S c$ as required.

Next we claim that if $a\sim_S b$ then $a$ and $b$ are twins. Let $x,y\in V\setminus\{a,b\}$ and suppose $axy\in G$. If $a\sim_S b$ via $cd$, then wlog $G[abcdxy]=S(ab,cx,dy)$ so $bxy\in G$. Similarly if $bxy\in G$ then $axy\in G$, so $a$ and $b$ are twins.

If $G$ has no edges then $G= S([n],\emptyset,\emptyset)$ so suppose $xyz\in G$. It is easy to check that $x,y,z$ are all in different equivalence classes, say $V_x,V_y,V_z$.  Moreover if $v\in V\setminus\{x,y,z\}$ then examining $G[vxyz]$ we see that $v\in V_x\cup V_y\cup V_z$. Finally, using the fact that related vertices are twins, we obtain $G=S(V_x,V_y,V_z)$.

For $\mc{P}_J$ the proof is very similar, the main difference being that we define $a\sim_J b$ iff $a=b$ or there exist distinct $c,d\in V\setminus\{a,b\}$ such that $G[abcd]=J(ab,cd)=\{abc,abd\}$. Again $\sim_J$ is an equivalence relation: if $a\sim_J b$ via $uv$ and $b\sim_J c$ via $xy$ but $a\not\sim_J c$ then wlog $c\neq u$ so $G[abcuvx]=J(abx,uvc)$ and $G[abcuxy]=J(bcu,axy)$,  but $acu$ is a non-edge in the former and an edge in the latter, a contradiction. We claim that if $a\sim_J b$ then $a$ and $b$ are twins. Let $x,y\in V\setminus\{a,b\}$ and suppose $axy\in G$. If $a\sim_J b$ via $cd$ then by examining $G[abcdxy]$ we see that $bxy\in G$. Similarly if $bxy\in G$ then $axy\in G$, so $a$ and $b$ are twins.

Now either $G$ is empty, so $G=J(\emptyset,V)$, or $G$ contains an edge. Let $xyz\in G$ and suppose that no two distinct vertices are related. Let $v\in V\setminus\{x,y,z\}$, then wlog $G[xyzv]=J(yzv,x)$. Now if $a,b\in V\setminus\{x\}$ are distinct then $G[xyzvab]=J(yzvab,x)$ so $abx\in G$ and hence $G=J(V\setminus\{x\},x)$. 

Finally let $V_x$ be a largest equivalence class with $x,y\in V_x$, $x\neq y$. If $\Gamma_{xy}=\{z:xyz\in G\}$ then it is straightforward to check that $V(G)=V_x\cup \Gamma_{xy}$ is a partition of $V(G)$ into independent sets. Thus, since all vertices in $V_x$ are twins, we have $G= J(V_x,\Gamma_{xy})$.


For $\mc{P}_B$ we define $a\sim_B b$ iff $a=b$ or there exist distinct $c,d,e\in V\setminus\{a,b\}$ such that $G[abcde]=B(ab,cde)$ (so the only non-edge in $G[abcde]$ is $cde$), in which case we say $a\sim_B b$ \emph{via} $cde$. Again we claim $\sim_B$ is an equivalence relation. Suppose $a\sim_B b$ via $uvw$ and $b\sim_B c$ via $xyz$, but $a\not\sim_B c$. Without loss of generality we may suppose that $a\not\in\{x,y\}$, $c\not\in\{u,v\}$ and $\{x,y\}\neq\{u,v\}$. So $G[abcuvw]=B(ab,cuvw)$ and $G[abcxyz]=B(bc,xyza)$. This implies that $G[abuvwx]=B(ab,uvwx)$ and $G[abuvwy]=B(ab,uvwy)$. But then we have $G[acuvxy]=B(ac,uvxy)$ so $a \sim_B c$. As before related vertices are twins: suppose $a\sim_B b$ via $cde$. If $x,y\in V\setminus\{a,b\}$ and $axy\in G$ but $bxy\not\in G$ then $G[abcde]=B(ab,cde)$ and $G[abxy]=B(a,bxy)$. Now wlog $G[abcdxy]=B(ac,bdxy)$, so $G[bcdexy]=B(\emptyset,bcdexy)$ is empty, a contradiction, since $bcd\in G$. Hence $a$ and $b$ are twins.

Let $V_x$ be a largest equivalence class. If $|V_x|\geq 2$ then suppose $x,y\in V_x$ are distinct. It is easy to check that  $G=B(V_x,\Gamma_{xy})$. If no equivalence class contains more than one vertex then either $G=B(\emptyset,V)$ is empty or one can check that there is a vertex $x$ such that $G=B(\{x\},V\setminus\{x\})$.
\end{proof}
We note that with a little extra work one can prove that all the 3-graph properties listed in Lemma \ref{ind:lem} are in fact 5-induced. 
\subsection{Exactness}
Although we have stability for most of our results we have not proved any exact Tur\'an numbers for the families we consider. In some, if not all cases, it may be possible to deduce an exact Tur\'an number  result from the stability theorem (along the same lines as Theorem 1 \cite{Pexact}) however the obvious approach to this would require a separate argument for each $\F$.
\section{Questions}
There are a number of obvious questions that arise from our work. 

We were able to show a number of exact results for single $3$-graphs with Tur\'an density $2/9,4/9,5/9$ and $3/4$. Is there a systematic way to find or predict such results? 
\begin{question}
Apart from blow-ups, are there any other operations under which Tur\'an densities are invariant?\end{question}

We were able to find a number of new Tur\'an densities of finite families of $3$-graphs by taking a small $3$-graph $G$ and then investigating the Tur\'an problem given by forbidding a family $\mc{F}_G$ of ``small'' $3$-graphs not contained in any blow-up of $G$. In most cases we were able to show that $\pi(\F_G)=\lambda(G)$. Does this hold more generally?
\begin{question}
Is $\Lambda^{(r)}\subseteq \Pi_{\tr{fin}}^{(r)}$?\end{question}

Having given the first examples of finite families with irrational Tur\'an densities we suspect there exist single $r$-graphs with irrational Tur\'an densities (indeed the pentagon $C_5$ is quite possibly an example see \cite{MR}, \cite{R4}).
\begin{question} Do there exist single $r$-graphs with irrational Tur\'an densities?\end{question}

Another natural question is the following:
\begin{question} For $r\geq 3$, which (if any) of the following containments between sets of densities are strict?
\[
\Pi^{(r)}_1\subseteq \Pi^{(r)}_2\subseteq \cdots \subseteq \Pi^{(r)}_{\tr{fin}}\subseteq \Pi^{(r)}_{\infty}.\]
\end{question}

Although we did not really consider ``induced Tur\'an problems'' here, we could certainly ask analogous questions about the associated sets of densities for induced Tur\'an problems.
\section{Computational proofs with flag algebras}\label{comp:sec}
Our proofs make use of Razborov's flag algebra framework introduced in \cite{RF}. In particular we follow the method outlined by Razborov in \cite{R4}. For a precise description we refer back to Section 2 of our previous paper \cite{BT} where we provided a self-contained and detailed explanation of the method. 

There are, however, two important ways in which the computations used to prove the results in this paper differ from our earlier work. Firstly we make extensive use of supersaturation, via Corollary \ref{sat:cor}. For example when computing the Tur\'an density of $H=\{123, 124, 345, 156\}$ (in Theorem \ref{2.9:thm}) we use the fact that if
\[
H_1=\{123, 124, 134\},\quad H_2=\{123, 124,125,345\}, \quad H_3=\{123, 124,135,245\}\]
then $H\leq H_i$ for $i=1,2,3$.

Thus applying Corollary \ref{sat:cor} (ii) we have $\pi(H)=\pi(\{H, H_1, H_2, H_3\})$.  This makes our computation significantly easier: there are 192 non-isomorphic 6 vertex 3-graphs that are $H$-free but only 38 of these are also $\{H_1,H_2,H_3\}$-free. (A very rough proxy for the difficulty of the computation is the final size of the proof file. In this case the use of Corollary \ref{sat:cor} (ii) reduces our proof to less than 15\% of the size of the smallest proof we could otherwise find. Moreover the computation completes in less than 10\% of the time it would otherwise take.) 

The second difference between the computations used to prove the results in this paper and those in \cite{BT} is that in this case we are proving exact sharp Tur\'an density results. 
Razborov already achieved this for the induced $\{K_4^{(3)},E_1\}$ problem (see Theorem \ref{raz:thm}) however in that case good use was made of the extremal construction to guide the conversion from numerical to exact result. We have found that  even without using any information about the extremal construction we can often identify the sharp inequalities and zero eigenvalues, and hence make numerical results exact.  In fact we have found that even for problems with many non-isomorphic extremal constructions we can sometimes prove exact Tur\'an density results (see Baber \cite{RBcube} for an example of this in the hypercube). This is of particular interest since Tur\'an's $K_4^{(3)}$ problem is a famous example where it is conjectured that there are many distinct extremal constructions \cite{K}, \cite{FdF}. However we note that when proving the irrational Tur\'an densities in Theorem \ref{fin:thm} we made extensive use of the extremal constructions. For more discussion of the process used to produce our proofs from numerical results see Section 2.4 in \cite{Rahil}. Details of how to check and reprove our results are in the next subsection.

Note, that in order to achieve the required accuracy when converting floating point numerical results into exact proofs in $\mathbb{Q}$ (or indeed $\mathbb{Q}[\sqrt{5}]$, $\mathbb{Q}[\sqrt{13}]$ for the irrational Tur\'an densities) we make use of arbitrarily large integers. Indeed a glance at the proof file \texttt{Root5-prf.txt} reveals integers with over 150 digits.


\subsection{Source code}
Although all of our proof files are ``human readable'' their size precludes verification by hand. However (with the exception of the irrational results) they can all be verified using the program \texttt{DensityChecker.cpp} \cite{DC}. 
We also attach the source code used to generate the majority of our proofs: \texttt{ExactDensityBounder.cpp} \cite{EDB}. This provides a simple command line program to give upper bounds for Tur\'an densities of $3$-graphs. It requires a semi-definite program solver: either \texttt{csdp} \cite{csdp} or \texttt{sdpa} \cite{sdpa}, both of which are freely available.

Detailed installation and usage instructions can be found in the source code files. We emphasise that these programs are both very easy to install and use.

\section{Acknowledgements}
We thank Oleg Pikhurko for drawing our attention to Conjecture \ref{CG:conj} which he has also disproved by very different methods \cite{PP}.

 As we were preparing this paper we learnt of a similar project by Falgas--Ravry and Vaughan \cite{EV} to prove exact results for 3-graphs. Their approach is quite similar to ours although the actual results we have obtained are distinct. We have been able to independently verify all of the proofs of exact results announced in their paper and we attach these as proof files: \texttt{FRV-12-49-prf.txt}, \texttt{FRV-5-18-prf.txt}, \texttt{FRV-3-4-prf.txt} and \texttt{FRV-3-8-prf.txt}. (We note that the three results for which they obtain the exact Tur\'an density $3/8$, Theorems 6, 7 and 13 in \cite{EV}, are in fact all implied by the single result $\pi(F_{3,2},F_{1,5})=3/8$, thus we only provide a single proof file in this case.)

We would like to thank a referee for very helpful comments and questions.


\end{document}